\documentclass[10pt]{article}

\usepackage[utf8]{inputenc}
\usepackage[T1]{fontenc}
\usepackage[english]{babel}
\usepackage[a4paper,text={15cm,23cm},centering]{geometry}
\usepackage{amsmath,amsthm}
\usepackage{libertinus}
\usepackage[libertine]{newtxmath}
\usepackage[varqu,varl,scaled=0.95]{zi4}
\usepackage[protrusion=true,expansion=true]{microtype}
\usepackage{enumitem,xcolor}
\usepackage[explicit]{titlesec}
\titleformat{\section}{\normalfont\sffamily\large\bfseries}{\thesection}{0.8em}{#1}
\titleformat{name=\section,numberless}{\normalfont\sffamily\large\bfseries}{}{0em}{#1}
\titleformat{\paragraph}[runin]{\normalfont\sffamily\bfseries}{}{0em}{#1}
\titlespacing*{\section}{0pt}{2.2ex plus .6ex minus .2ex}{1.1ex plus .2ex}
\titlespacing*{\paragraph}{0pt}{1.4ex plus .4ex minus .2ex}{0.6em}
\usepackage{fancyhdr}
\definecolor{linkblue}{rgb}{0.10,0.25,0.55}
\usepackage[colorlinks,linkcolor=linkblue,citecolor=linkblue,urlcolor=linkblue]{hyperref}

\newtheorem{theorem}{Theorem}[section]
\newtheorem{lemma}[theorem]{Lemma}

\theoremstyle{definition}
\newtheorem{definition}[theorem]{Definition}
\newtheorem{question}[theorem]{Question}
\theoremstyle{remark}

\setlist[itemize]{parsep=0pt,topsep=2pt}
\setlist[enumerate]{parsep=0pt,topsep=2pt}

\newcommand{\N}{\mathbb{N}}
\newcommand{\Prop}{\mathsf{Prop}}
\newcommand{\Type}{\mathsf{Type}}
\newcommand{\ZF}{\mathrm{ZF}}
\newcommand{\ZFC}{\mathrm{ZFC}}
\newcommand{\CIC}{\mathrm{CIC}}
\newcommand{\EM}{\mathrm{EM}}
\newcommand{\Con}{\mathrm{Con}}
\newcommand{\rk}{\operatorname{rank}}
\newcommand{\Acc}{\mathsf{Acc}}
\newcommand{\V}{\mathbb{V}}
\newcommand{\la}{\mathsf{a}}
\newcommand{\lb}{\mathsf{b}}
\newcommand{\lc}{\mathsf{c}}
\newcommand{\lean}[1]{\texttt{#1}}
\newcommand{\guard}[2]{\{#1\mid #2\}}

\begin{document}

\title{\sffamily\bfseries CIC + EM $\boldsymbol{\vdash}$ Con(ZF):\\[0.3ex]
  \normalfont\sffamily\large the consistency of ZF in type theory with excluded middle and no choice}
\author{Mario Carneiro\\[0.5ex]
  \small Chalmers University of Technology and University of Gothenburg\\
  \small Gothenburg, Sweden\\
  \small\href{mailto:marioc@chalmers.se}{\texttt{marioc@chalmers.se}}}
\date{\today}
\maketitle

\begin{abstract}
The sets-as-trees interpretation of set theory in a dependent type
theory with an impredicative universe of propositions validates
Zermelo set theory, and it validates Replacement if the type theory
has a choice or description operator, which turns a functional relation
into a function. It has been natural to expect that without such an
operator the strength of the type theory drops well below that of $\ZF$. We
show that it does not. In the type theory of Lean with two predicative
universes, from excluded middle as the only assumption and with no
axiom (no choice, no propositional extensionality, no quotients), we
prove the consistency of $\ZF$, stated outright for a first-order proof
system. The proof is formalized. The mechanism is the large elimination
of the accessibility predicate over a type as large as the type of
sets: a recursion on accessibility whose recursive calls are guarded by
propositions, and whose later calls are indexed by the \emph{value} of
an earlier call, computes as a term any ordinal that is specified by a
proposition through a well-founded tree of a certain shape. We give a
rule that produces such a tree for every ordinal, unless some $V_\rho$ is
already a model of $\ZF$; the rule does not choose a cofinal map into a
limit ordinal but takes all definable ones at once. In the first case
the sets-as-trees satisfy Replacement for arbitrary propositional
relations. Either way $\ZF$ has a model. Finally, the double negation of
excluded middle suffices, and what remains of it is exactly that
membership is not not well-founded in the stable reading of sets; this
in turn implies the double negation of Markov's principle.
\end{abstract}

\section{Introduction}

\paragraph{The question.}
Consider a dependent type theory with an impredicative, proof-irrelevant
universe $\Prop$, predicative universes $\Type_0:\Type_1:\cdots$,
inductive types, and an accessibility predicate $\Acc$ in $\Prop$ that
eliminates into every universe: the Calculus of Inductive Constructions
in the form implemented by Lean \cite{lean4,carneiro19}. How strong is
it, as a foundation, when excluded middle is assumed but no choice
principle is?

With choice the answer is known. Aczel's interpretation of sets as
well-founded trees \cite{aczel78} gives a type $\V$ of sets in which
the axioms of Zermelo set theory hold, by impredicativity of $\Prop$ for
Separation and Power set. Replacement asks for the image of a set $s$
under a functional relation $\varphi$. A set in $\V$ is a tree
$\mathsf{sup}\,A\,f$ with $A$ a small type and $f:A\to\V$ a
\emph{function}, so to form the image one needs a function $a\mapsto y_a$
with $\varphi(a,y_a)$, and a functional relation in $\Prop$ does not
provide one, because propositions do not eliminate into types. A choice
or description operator closes the gap, and with it the type theory
with $n+2$ universes interprets $\ZFC$ with $n$ inaccessible cardinals
\cite{werner97,barras10,carneiro19,kirstsmolka18,kirstsmolka19}. Without it, the gap between a
functional relation and a function looks like exactly the gap between
Zermelo set theory and $\ZF$, and the standard set-theoretic model, in
which the universes are $V_\kappa$ for inaccessible $\kappa$, does not tell
whether the inaccessibles are needed.

\paragraph{Unique choice.}
The gap is the principle of unique choice: a proof of
$\forall x{\in}s\,\exists!y.\ \varphi(x,y)$ should yield a function, and
whether it does depends on how propositions are read. In Martin-L\"of
type theory, where propositions are types, the existential is a
$\Sigma$-type, a proof of $\forall x\,\exists y\,\varphi$ already is a function,
and Aczel's model validates Replacement and even Strong Collection
\cite{aczel78}; but there is no impredicative universe, and the
strength of these type theories is far below that of $\ZF$
\cite{griffor94}. In homotopy type theory the propositions are the types
with at most one element, and unique choice is provable, because the
type of witnesses of a unique existence is itself a proposition. The
cumulative hierarchy there satisfies Replacement, and with excluded
middle it is a model of $\ZF$: the proof of \cite[Theorem~10.5.11]{hottbook}
uses choice only to obtain excluded middle (see also \cite{gylterud18}
for the same type built from $W$-types and univalence). In the Calculus
of Inductive Constructions, propositions form an impredicative universe
that does not eliminate into types, and Rocq and Lean provide unique
choice only as an axiom, of description or of choice; this is what the
models cited above assume. Without it the strength can be low: Martin-L\"of
type theory with one universe, Aczel's type of sets, and an impredicative,
proof-irrelevant type of propositions has the strength of Zermelo set
theory together with the existence of $V_\tau$ for every $\tau$ below the
Bachmann--Howard ordinal \cite{rathjen12}. That type theory has neither
excluded middle nor propositions that eliminate into types.

\paragraph{The result.}
They are needed, in the sense that the type theory with two universes
and excluded middle, without choice, already proves the consistency of
$\ZF$.

\begin{theorem}[formalized]\label{thm:main}
In the type theory of Lean~4, using the universes $\Prop$, $\Type_0$ and
$\Type_1$ and no axioms, excluded middle implies $\Con(\ZF)$, where $\ZF$
is the usual first-order theory with the schemas of Separation and
Replacement and $\Con$ is non-derivability of falsity in a Hilbert
system. More precisely, excluded middle implies that one of the
following holds.
\begin{enumerate}
\item Every functional relation $\varphi:\V\to\V\to\Prop$ on a set $s:\V$ has
  an image in $\V$. Then $\V$ is a model of $\ZF$, in fact of its
  second-order form.
\item There is an ordinal $\rho>\omega$ of $\V$, a limit, such that $V_\rho$ is
  closed under the images of functions definable over $V_\rho$ with
  parameters. Then $V_\rho$ is a model of $\ZF$.
\end{enumerate}
\end{theorem}

By G\"odel's second incompleteness theorem, $\ZF$, if consistent, does
not prove the consistency of this type theory, with or without choice,
so no model construction carried out in $\ZF$ can succeed for it.

\paragraph{The mechanism.}
Propositions do not eliminate into types, with one family of exceptions:
inductive propositions with at most one constructor whose arguments are
propositions or are determined by the indices. Accessibility is the
important one. For a relation $R$ on a type $X$, $\Acc_R\,x$ has the one
constructor $(\forall y.\ R\,y\,x\to\Acc_R\,y)\to\Acc_R\,x$, and its eliminator
defines a function on the accessible elements of $X$ with values in any
type, by recursion on $R$. This is what makes well-founded recursion
available, and it is known to extract information from propositions:
the standard library of Rocq obtains a witness of $\exists n.\,P\,n$, for
decidable $P$ on $\N$, by recursion on a proof of an inductive
proposition of the same kind \cite{constructiveepsilon}. That is unbounded search, and it needs the
decidability of $P$ to know when to stop.

We use the same idea transfinitely, with the type of sets itself as the
type of values, and two observations replace decidability.
\begin{itemize}
\item \emph{Guards.} For a proposition $P$ and $t:P\to\V$ there is a set
  $\guard{t}{P}$ which is $t$ if $P$ holds and empty otherwise, namely
  the union of the family $t$ indexed by the proofs of $P$. No decision is made:
  the index type of the tree is $P$ itself.
\item \emph{Adaptive calls.} A step of the recursion may call the
  recursion at a child, obtain a set $G$ as the value, and then call it
  at further children indexed by the elements of a set computed from
  $G$. The index types of sets are small, but there is no bound on the
  sets that occur as values, so the trees that such a recursion explores
  are not bounded in advance by anything the recursion starts from.
\end{itemize}
The \emph{shape} of the tree explored by the recursion, that is, which
guards hold, is given by a relation $R$ on paths which is an arbitrary
proposition. If $R$ is the tree of a well-founded specification of an
ordinal $\eta$, in a sense made precise below, then the root is accessible
and the value of the recursion at the root is $\eta$: an ordinal that was
only specified becomes the value of a term
(Lemma~\ref{lem:mat}). Applied to the ranks of the values of a
functional relation on a set $s$, with one child of the root for each
element of $s$, the value at the root bounds all the ranks, and the
image is obtained by Separation (Theorem~\ref{thm:repl}).

What remains is to specify every ordinal by such a tree, uniformly,
because the trees for the different elements of $s$ must be given by
one proposition. At a limit ordinal $\eta$ the children must have targets
cofinal in $\eta$ and be indexed by a set of rank below $\eta$. We cannot
choose a cofinal map. Instead the children are indexed by \emph{all}
triples of a first-order definition with parameters, a domain, and an
argument, such that the definition gives over $V_\eta$ a function on the
domain with values of unbounded rank; the target of a child is the rank
of the value at the argument. This fails at $\eta$ exactly if no definable
function from a set in $V_\eta$ is unbounded in $V_\eta$, and then $V_\eta$ is a
model of $\ZF$ (Section~\ref{sec:rule}).

\paragraph{The carrier must be big.}
The recursion runs on paths whose labels are arbitrary sets, so its
carrier lives in the same universe $\Type_1$ as $\V$. This is essential.
On a carrier $X$ in $\Type_0$ the height of any accessible element is
below the Hartogs number of the power set of $X$, which is an ordinal of
$\V$ obtained by a term without Replacement, so the tree is bounded in
advance.
We do not know the strength of the type theory in which accessibility is
restricted to carriers in the lowest universe
(Question~\ref{q:small}).

\paragraph{Formalization.}
The development is in Lean~4 without Mathlib, about 2000 lines, and
is described in Section~\ref{sec:lean}; it is available at
\url{https://github.com/digama0/ConZF}, and the version described here
is the tag \lean{arxiv-v1}. The statement of the main
theorem is
\begin{verbatim}
theorem PSet.con_ZF (em : forall p : Prop, p \/ Not p) : Con ZF
\end{verbatim}
and Lean reports that it depends on no axioms. Everything in
Sections~\ref{sec:sets} to~\ref{sec:zf} is formalized; the text follows
the formal development and names the corresponding declarations.
Section~\ref{sec:em} reports on removing excluded middle.

\section{Sets as trees}\label{sec:sets}

\paragraph{The type theory.}
We use $\Prop$ with its impredicative quantifier, two predicative
universes, inductive types with structural recursion, and $\Acc$ with
elimination into $\Type_1$. The inductive types used are natural numbers,
lists, subtypes, the lifting of a proposition or a small type to a
universe, a three-constructor type of labels, the syntax of formulas,
and the type $\V$ below; all of them are instances of $W$-types and
finite sums. Excluded middle is a hypothesis
$\mathit{em}:\forall p{:}\Prop.\ p\vee\neg p$ of the theorems that need
it. Proof irrelevance, propositional and function extensionality, and
quotients are not used.

\begin{definition}[\lean{PSet}]
$\V:\Type_1$ is the inductive type with the constructor
$\mathsf{sup}:\Pi A{:}\Type_0.\ (A\to\V)\to\V$. For $x=\mathsf{sup}\,A\,f$ we write
$\bar x:=A$ for the \emph{index type} and $x_i:=f\,i$. Bisimulation
$x\approx y$ is defined by recursion:
$\forall i\,\exists j.\ x_i\approx y_j$ and $\forall j\,\exists i.\ x_i\approx y_j$.
Membership is $x\in y:=\exists j.\ x\approx y_j$.
\end{definition}

Bisimulation is an equivalence relation, membership respects it on both
sides, $x\approx y$ iff $x$ and $y$ have the same elements, and
$\in$-induction holds for every predicate on $\V$, whether or not it
respects $\approx$. We never quotient: a \emph{set} is an element of $\V$,
and the statements below are about $\approx$.

Think of $\mathsf{sup}\,A\,f$ as a well-founded tree whose root has one
subtree $f\,a$ for each $a:A$. Two trees are bisimilar if they differ
only in the order and multiplicity of subtrees, hereditarily, and a set
in the usual sense is a tree up to bisimilarity. The index type $\bar x$
is part of the data of $x$: a set is presented as an enumeration of its
elements, indexed by a small type.

\paragraph{Operations.}
All of the following are terms, with the expected membership
conditions. The empty set; $\{A_i\mid i:\iota\}$ for a family indexed
by a small type; $\bigcup x$, and $\bigcup_{i:\iota}A_i$; pairs
$\{x,y\}$, the successor $x\cup\{x\}$, Kuratowski pairs $\langle x,y\rangle$, which
are injective up to $\approx$, and tuples; Separation
$\{z\in x\mid P\,z\}$ for an arbitrary $P:\V\to\Prop$, with index type the
subtype $\{i:\bar x\mid P\,x_i\}$; the power set, with index type
$\bar x\to\Prop$, which is small because $\Prop$ is impredicative;
$\omega$; and, for a proposition $P$ and $t:P\to\V$, the \emph{guard}
\[\guard tP:=\textstyle\bigcup_{h:P}t\,h,\qquad
z\in\guard tP\iff\exists h{:}P.\ z\in t\,h.\]
The guard is a set-valued `if $P$ then $t$ else $\emptyset$' whose condition is
never evaluated: when $P$ is false the union is over an empty index
type. It is the one way in which a proposition enters a set without
being decided, and it is used throughout.

So $\V$ satisfies Zermelo set theory with Foundation, with Separation
for all propositional predicates. With excluded middle, ordinals
(transitive sets of transitive sets) are linearly ordered by $\in$. The
rank $\rk x:=\bigcup_i(\rk x_i\cup\{\rk x_i\})$ is a term, it is an
ordinal, and $\rk\eta\approx\eta$ for an ordinal $\eta$. The levels are given by
the term $V_x:=\bigcup_i\mathcal P(V_{x_i})$, defined for every set
$x$, with $y\in V_x$ iff $\rk y\in\rk x$. Note that $\rk$ and $V$ are
functions because they are defined by recursion on \emph{data}. What
$\V$ lacks is any way to obtain a set from a description of it. A
function into $\V$ must build the index type and the enumeration of its
result, and a proof that exactly one set $y$ satisfies $\varphi(x,y)$ contains
no such enumeration, since propositions do not eliminate into types.
This is why Replacement fails to be evident, and the next section is
about the one exception we have.

\section{The materializing recursion}\label{sec:mat}

This section and the next two carry out the plan of the introduction.
Here a well-founded specification of an ordinal is turned into a term;
Section~\ref{sec:repl} uses this to prove Replacement from a supply of
specifications; and Section~\ref{sec:rule} produces the specifications.

\paragraph{The idea.}
Suppose an ordinal $\eta$ is given only by a proposition, and we want a
term of type $\V$ that is bisimilar to it. An ordinal is the set of the
smaller ordinals, but it is equally the union of the successors
$\xi\cup\{\xi\}$ over any set of ordinals $\xi$ below $\eta$ that is cofinal in
$\eta$, or that contains the predecessor if $\eta$ is a successor. So $\eta$
could be computed by recursion, if a node standing for $\eta$ could name
some children whose values are cofinal below $\eta$, compute those values
recursively, and return the union of their successors. The values need
not be all of $\eta$'s elements, only enough of them. The difficulty is
the naming: the recursion calls itself on the children, so they must be
data, and the recursion must produce the names without knowing $\eta$.

Take a natural number $n$ first. The node makes one call, at a child
named $\la$, and returns the successor of the value. Whether the child
should exist, that is, whether $n>0$, is a proposition that the
recursion cannot decide. So the shape of the tree is given by a relation
$R$ on nodes, an arbitrary proposition, and the call at a child $c$ of $p$
is guarded by $R\,c\,p$. If the guard fails the value is $\emptyset$, and
otherwise it is the successor of the value at the child. For $\omega$ one
call is not enough, and the node also calls children $\lb\,w$ for all
$w\in V_\omega$, a set given by a term (below, it is $D(\emptyset)$); the child $\lb\,m$ for $m\in\omega$ has the
value $m$, and the union of the successors is $\omega$. For an arbitrary
limit $\eta$ no fixed supply of names works, since a cofinal family must
be indexed by a set that grows with $\eta$. The node therefore first calls
the child $\la$ and uses its value $G$, an ordinal below $\eta$ but possibly
large, to compute a supply $D(G)$ of names, and then calls a child $\lb\,w$
for every $w\in D(G)$. It is like a program that first computes a bound
by a recursive call and then allocates an array of that size, filling
each entry by another recursive call. The relation $R$ says which of
these children exist, and so which of the names are used.

A third kind of label, $\lc\,x$ for $x$ in a set $U$ fixed in advance, lets
a node branch over a set that it does not compute. It is used only at
the root, in the proof of Replacement, where $U$ is the domain of the
relation and the root has one child for each of its elements.

A node is named by its path from the root, a list of labels. The
recursion is on accessibility for $R$ on paths, so it may run through a
tree of any height; the targets, the ordinals the nodes stand for, are
not data and play no role in the computation.

\paragraph{The recursion.}
Fix a term $D:\V\to\V$ and a set $U:\V$. In the application $D(G)$
is $V_{\rk G+\omega}$ and $U$ is the domain of a functional relation.

\begin{definition}[\lean{Label}, \lean{Path}, \lean{step}, \lean{F}]\label{def:mat}
A \emph{label} is $\la$, or $\lb\,w$, or $\lc\,x$, for sets $w,x$; two
labels are equivalent if they have the same constructor and
bisimilar arguments. A \emph{path} is a list of labels, deepest first;
$[\,]$ is the root and $l::p$ is the child of $p$ with label $l$. Labels
and paths are types in $\Type_1$. For a relation $R$ on paths, a path $p$
and $\mathrm{rec}:\Pi c.\ R\,c\,p\to\V$ let
\begin{align*}
\mathrm{call}(l)&:=\guard{\mathrm{succ}\,(\mathrm{rec}\,(l::p)\,h)}{h:R\,(l::p)\,p},\\
G&:=\guard{\mathrm{rec}\,(\la::p)\,h}{h:R\,(\la::p)\,p},\\
\mathrm{step}_R\,p\,\mathrm{rec}&:=\mathrm{call}(\la)\ \cup
  \textstyle\bigcup_{i:\overline{D(G)}}\mathrm{call}(\lb\,D(G)_i)\ \cup
  \bigcup_{j:\bar U}\mathrm{call}(\lc\,U_j),
\end{align*}
and let $F_R:\Pi p.\ \Acc_R\,p\to\V$ be defined from $\mathrm{step}_R$ by the
eliminator of $\Acc$, so that
\[F_R\,p\,a=\mathrm{step}_R\,p\,(\lambda c\,h.\ F_R\,c\,a_{c,h}),\]
where $a_{c,h}:\Acc_R\,c$ is obtained from $a$ and $h:R\,c\,p$.
\end{definition}

The step calls the recursion at the child $\la$; then at the children
$\lb\,w$ for the elements $w$ of $D(G)$, where $G$ is the value of the
first call; and at the children $\lc\,x$ for $x\in U$. Every call is
guarded by the proposition that the child is below $p$, which the step
cannot prove and does not need to decide. The value is the union of the
successors of the values of the calls whose guard holds.

\paragraph{Specifications.}
The recursion computes something for every $R$, and to say what, we
describe the intended value of each node by a proposition. A
\emph{target assignment} says which set each path is meant to compute. It
plays the role of the specification of a recursive program, and
coherence is the verification condition saying that one step meets the
specification if the recursive calls do. A path with no target is a
node that should not exist, and the shape of the tree is read off from
the assignment: a child exists iff it has a target. Targets are given
by a relation, not a function, because they are specified by
propositions.

\begin{definition}[\lean{Coherent}]\label{def:coherent}
A \emph{target assignment} is a relation $\tau$ between paths and sets.
We say that $p$ \emph{has the target} $t$ if $\tau\,p\,t$. Its relation is
$R_\tau\,c\,p$: $c=l::p$ for some $l$, and $c$ has a target. Say that $G$
\emph{is the first value at $p$} if the elements of $G$ are the elements
of the targets of $\la::p$ (so $G$ is empty if $\la::p$ has no target),
and that a label $l$ is \emph{available for $G$} if $l=\la$, or $l$ is
equivalent to $\lb\,w$ with $w\in D(G)$, or to $\lc\,x$ with $x\in U$.
The assignment is \emph{coherent} if
\begin{enumerate}
\item the targets of a path are closed under $\approx$, and any two are
  bisimilar;
\item equivalent labels give children with the same targets;
\item if $l::p$ has the target $t$ and $p$ has the target $t'$ then
  $t\in t'$; and
\item if $p$ has the target $t$ and $G$ is the first value at $p$, then
  $x\in t$ iff $x\in t'\cup\{t'\}$ for some target $t'$ of a child $l::p$
  with $l$ available for $G$.
\end{enumerate}
\end{definition}

In words: (1) the target of a path is determined up to $\approx$; (2)
labels that are equal as sets name the same node; (3) targets decrease
along the tree, which, like a decreasing measure for a recursive
program, will give accessibility by $\in$-induction; and (4) if every
child returns its target, the step returns the target of the parent. In
(4), $G$ is what the first call returns and the available labels are
those at which the step calls. A coherent assignment is a proposition
about a relation. No part of it is data, and it may be defined using any
quantifiers. Condition (3) does not require the parent of a path with a
target to have one.

\begin{lemma}[materialization; \lean{materialize}]\label{lem:mat}
If $\tau$ is coherent and $p$ has the target $t$, then $p$ is
$R_\tau$-accessible, and $F_{R_\tau}\,p\,a\approx t$ for every $a:\Acc_{R_\tau}\,p$.
\end{lemma}
\begin{proof}
By $\in$-induction on $t$, for all $p$. If $R_\tau\,c\,p$ then $c=l::p$ has a
target $t_c$, and $t_c\in t$ by (3); so the induction hypothesis applies
to every such $c$: it is accessible, and the recursion returns $t_c$ at
it. Hence $p$ is accessible, by the constructor. For the value, unfold
$F\,p\,a$ once. A guard at a child holds iff the child has a target, and
then the recursive call returns that target, up to $\approx$. So the set
$G$ of the step is the first value at $p$; the labels at which the step
makes calls are, up to equivalence, the labels available for $G$, and by
(2) equivalence does not matter; and the value of the step has as
elements the $x\in t'\cup\{t'\}$ for the targets $t'$ of the children with
available labels. By (4) this is $t$.
\end{proof}

The proof uses no excluded middle. The height of the recursion is not
bounded by anything computed from $p$: if the child $\la$ has value $G$,
the children $\lb\,w$ range over $D(G)$, which may be much larger than
anything seen so far, and so on below them.

\section{Replacement from uniform assignments}\label{sec:repl}

Let $\varphi$ be functional on a set $s$, and suppose that for each $x\in s$ the
rank $\eta_x$ of the value at $x$ has a coherent assignment specifying it. We
want a set containing all the $\eta_x$. Build one tree: the root has a
child $\lc\,x$ for each $x\in s$, and below that child is the tree
specifying $\eta_x$. The root itself gets no target, because its target
would be the supremum of the $\eta_x$, whose existence is what we are
proving. It does not need one. Accessibility of a node only asks for
accessibility of its children, so the root is accessible, the recursion
returns a term there, and unfolding one step shows that the term is the
supremum. The specification is needed below the root but not at it.

One point needs care. The combined tree must be described by a single
proposition, which says of a path through $\lc\,x$ that $x\in s$, that $\eta$ is
the rank of the value at $x$, and that the rest of the path has a given
target under the assignment for $\eta$. So the assignments must be given
for all $\eta$ at once, as one relation $T_\eta$ with $\eta$ as a parameter: we
cannot pick an assignment for each $x$, since that would be a choice
over $s$. This is the role of the family $T$ in the theorem.

\begin{theorem}[\lean{replacement}]\label{thm:repl}
Let $C$ be a class of sets and $T_\eta$ a target assignment for each set
$\eta$, such that $T_\eta$ only depends on $\eta$ up to $\approx$ and, for $\eta$ in $C$,
$T_\eta$ is coherent for the parameter $U=s$ and gives the root the
target $\eta$. Let $\varphi$ be a relation that respects $\approx$ and is functional
on $s$, with all values in $C$. Then there is a set whose elements are
the $y$ with $\varphi(x,y)$ for some $x\in s$.
\end{theorem}
\begin{proof}
Let $\tau$ be the assignment that gives the path $r\mathbin{+\!\!+}[\lc\,x]$ the
target $t$ if $x\in s$, $\varphi(x,\eta)$, and $T_\eta\,r\,t$. It is \emph{one}
proposition, defined from $s$, $\varphi$ and $T$; no assignment is chosen for
each $x$. It is coherent, each condition reducing to the same condition
for some $T_\eta$, since the paths below $[\lc\,x]$ are in bijection with all
paths. The root has no target under $\tau$. Its target would be the union
of the $\eta\cup\{\eta\}$, and the existence of that set is the point at
issue. But every child of the root that has a target is accessible by
Lemma~\ref{lem:mat}, so the root is accessible by the constructor, and
$\theta:=F_{R_\tau}\,[\,]\,a$ is a term of type $\V$, because $s$ is a
variable of type $\V$. Unfolding the recursion once, the first value at
the root is empty, the children with a target are the $[\lc\,x]$ with
$x\in s$, and by the lemma the calls there return the values of $\varphi$. So
$y\in\theta$ iff $y\in\eta\cup\{\eta\}$ for some value $\eta$ of $\varphi$ on $s$, and the image is
$\{y\in\theta\mid\exists x\in s.\ \varphi(x,y)\}$.
\end{proof}

The assignments will come from a rule that acts at one node at a time.
A rule says, given the target $\eta$ of a node, which labels have children
and what their targets are; unfolding it from the root gives an
assignment, uniformly in $\eta$. Think of it as a strategy for descending
from an ordinal: at each node, name some smaller ordinals that are
cofinal below it, using names the step can produce from the first
value.

\begin{lemma}[\lean{Rule.coherent}]\label{lem:rule}
Let $r(\eta,l,\xi)$ be a relation, the \emph{rule}, which determines $\xi$ up
to $\approx$ from $\eta$ and $l$ and respects $\approx$ and equivalence of labels. Let
$C$ be a class closed under $\approx$ such that for $\eta$ in $C$: if $r(\eta,l,\xi)$
then $\xi\in\eta$ and $\xi$ is in $C$; and, for $G$ with elements the
elements of the $\zeta$ with $r(\eta,\la,\zeta)$, we have $x\in\eta$ iff
$x\in\xi\cup\{\xi\}$ for some $l$ available for $G$ and $\xi$ with $r(\eta,l,\xi)$.
Define $T_\eta$ by recursion on the path: the root has the targets
bisimilar to $\eta$, and $l::p$ has the target $\xi$ if $p$ has a target $t$
with $r(t,l,\xi)$. Then $T_\eta$ is coherent for every $\eta$ in $C$.
\end{lemma}
\begin{proof}
Every target is in $C$, by induction on the path, and the conditions of
Definition~\ref{def:coherent} at $p$ are the hypotheses on the rule at the
target of $p$.
\end{proof}

\section{The definability rule}\label{sec:rule}

We need a rule whose class $C$ contains every ordinal. At a successor
$\zeta+1$ one child with target $\zeta$ suffices, and at $\omega$ the children can
be indexed by the natural numbers. At any other limit $\eta$ the children
must have targets cofinal in $\eta$, their labels must come from $D(G)$,
and $G$, being the target of a child, is below $\eta$. So the rule must
send sets of rank below $\eta$ cofinally into $\eta$, definably in $\eta$. We
know of no definable way to select one cofinal map. The evident way uses
a definable well-order of the universe, and $\ZF$ does not provide one:
it is consistent, even with the axiom of choice, that there is no
definable well-order of the reals \cite{feferman65}. (A first version of
the construction worked inside G\"odel's $L$, where there is such a
well-order, selected the least cofinal map, and obtained a model of
$\ZFC$ inside $L$.) The solution is not to select, and the rule below
needs neither $L$ nor choice.

Instead the node at $\eta$ takes \emph{every} cofinal map of a certain
kind, one that can be named by data of rank below $\eta$; whichever map one
would have liked to choose, it is among them. A map is named by a set
$q$, a code, together with its domain $s$. The codes are interpreted by a
single relation $I$ fixed in advance: $I(\eta,q,x,y)$ says that the
function named by $q$, evaluated over $V_\eta$, sends $x$ to $y$. It is an
evaluator whose programs are sets. The rule works for any $I$, and in
Section~\ref{sec:zf} it is first-order satisfaction in $V_\eta$, with $q$
coding a formula and its parameters. That choice makes the failure of
the rule useful: if no function named by a code of rank below $\eta$ is
unbounded in $\eta$, then $V_\eta$ satisfies Replacement for every formula.

The first call supplies the names. The label $\lb\,w$ at $\eta$ is a triple
$\langle q,s,x\rangle$: a code, a domain, and a point of the domain, with the
rank of the value at $x$ as its target. The first value $G$ must be large
enough that $D(G)=V_{\rk G+\omega}$ contains these triples, and, being a
target, it must be determined by $\eta$. The least ordinal $\nu$ such that
some code and domain of rank below $\nu$ give an unbounded map has both
properties.

\begin{definition}[\lean{Unb}, \lean{Reach}, \lean{G$\nu$}]
Fix a relation $I(\eta,q,x,y)$ respecting $\approx$, read as ``$y$ is the value
at $x$ of the function defined over $V_\eta$ by the code $q$''. A pair
$(q,s)$ is \emph{unbounded at $\eta$} if $I(\eta,q,\cdot,\cdot)$ is functional on
$s$, every value at an $x\in s$ has rank in $\eta$, and for every $\zeta\in\eta$
some value at an $x\in s$ has rank at least $\zeta$. The ordinal $\eta$ is
\emph{reached from $\nu$} if some pair $(q,s)$ with $\rk q,\rk s\in\nu$ is
unbounded at $\eta$, and \emph{reachable} if it is reached from some
$\nu\in\eta$. For reachable $\eta$ let
$\nu(\eta):=\{\nu\in\eta\mid\eta\text{ is not reached from }\nu\}$, which is the
least ordinal from which $\eta$ is reached.
\end{definition}

The least $\nu$ is obtained by Separation and not by choosing: being
reached from $\nu$ is upwards closed in $\nu$, so the set of the $\nu$ that
fail is an ordinal, it is an element of $\eta$, and with excluded middle
$\eta$ is reached from it.

\begin{definition}[\lean{rule}]\label{def:rule}
With $D(G):=V_{\rk G+\omega}$, the rule $r(\eta,l,\xi)$ holds in the following
cases and no others.
\begin{enumerate}
\item $\eta\approx\zeta\cup\{\zeta\}$, $l=\la$, and $\xi\approx\zeta$.
\item $\eta\approx\omega$, and either $l=\la$ and $\xi\approx\emptyset$, or $l=\lb\,w$ with
  $\rk w\in\omega$ and $\xi\approx\rk w$.
\item $\eta$ is not a successor, not $\omega$, and reachable, and either
  $l=\la$ and $\xi\approx\nu(\eta)$, or $l=\lb\,w$ with $w\approx\langle q,s,x\rangle$, where
  $(q,s)$ is unbounded at $\eta$, $x\in s$, and $\xi$ is the rank of the value
  at $x$.
\end{enumerate}
\end{definition}

\paragraph{Example.}
Take $\eta=\omega+\omega$. It is not a successor and not $\omega$. The formula
$\psi(x,y)$ saying that $x\in\omega$ and $y=\omega+x$ defines over $V_\eta$ a function
on $\omega$ with values cofinal in $\eta$. Its code $q$ has finite rank and
$\omega$ has rank $\omega$, so $\eta$ is reached from $\omega+1$; it is reached from no
smaller ordinal, since a function on a set of finite rank has finitely
many values. So the root has the child $\la$ with target $\omega+1$, and a
child $\lb\,\langle q',s',x\rangle$ for every unbounded pair $(q',s')$ and $x\in s'$. The
step calls those named in $D(\omega+1)$, among them $\lb\,\langle q,\omega,n\rangle$ with
target $\omega+n$. Below a node with target $\omega+n$ a chain of children $\la$
descends to $\omega$, whose children $\lb\,w$, for $w$ of finite rank, have the
finite targets $\rk w$. The
recursion never sees $\eta$. It computes $\omega+1$ at the child $\la$ of the
root, then $V_{\omega+1+\omega}$, then calls every child named by an element of
it, and the tree tells it which of those calls succeed.

\begin{lemma}[\lean{worldly\_rule}]\label{lem:worldlyrule}
Assume excluded middle. Let $C$ be the class of ordinals $\eta$ such that
every ordinal $\mu\le\eta$ is $0$, a successor, $\omega$, or reachable. The rule
satisfies the hypotheses of Lemma~\ref{lem:rule} for $C$.
\end{lemma}
\begin{proof}
The cases are exclusive, and in each the target is determined by $\eta$
and the label: in (3) because a triple determines its components and
$(q,s)$ is functional on $s$. Every target is an element of $\eta$, and
$C$ is closed downwards. For the covering condition let $x\in\eta$. In (1),
$x\in\zeta\cup\{\zeta\}$ and $\la$ is available. In (2), $x$ is a natural number
$n$, the first value is $\emptyset$, and $n\in V_\omega=D(\emptyset)$. In (3),
$\eta$ is reached from $\nu(\eta)$, by a pair $(q,s)$; there is $x'\in s$ with a
value $y$ such that $x\in\rk y\cup\{\rk y\}$; the first value is $\nu(\eta)$; and
$\langle q,s,x'\rangle$ has rank below $\nu(\eta)+4$, since $q$, $s$ and $x'\in s$ have
rank below $\nu(\eta)$, so it is in $D(\nu(\eta))$. If $\eta$ is $0$ it has no
elements, and no child has a target.
\end{proof}

\begin{theorem}[\lean{dichotomy}]\label{thm:dich}
Assume excluded middle. Either every relation $\varphi$ that respects $\approx$ and
is functional on a set $s$ has an image in $\V$; or there is an ordinal
that is not $0$, not a successor, not $\omega$, and not reachable.
\end{theorem}
\begin{proof}
If every ordinal is $0$, a successor, $\omega$, or reachable, then $C$ is
the class of all ordinals. Given $s$ and $\varphi$, apply
Theorem~\ref{thm:repl} to the relation ``$\eta$ is the rank of the value of
$\varphi$ at $x$'', whose values are ordinals, with the assignments of
Lemmas~\ref{lem:rule} and~\ref{lem:worldlyrule}. Its image is a set of
ordinals, which is included in an ordinal $\theta$ by the proof of the
theorem, and the image of $\varphi$ is separated from $V_\theta$.
\end{proof}

The theorem holds for any $I$. The first alternative says more the
fewer ordinals are reachable, the second says more the more are; we now
choose $I$ so that the second gives a model of $\ZF$.

\section{\texorpdfstring{$\ZF$}{ZF} and its two models}\label{sec:zf}

\paragraph{Syntax and proofs (\lean{Fml}, \lean{Prf}, \lean{Con}).}
Formulas are built from $v_i\in v_j$, $v_i=v_j$, $\bot$, $\to$ and $\forall$,
with de Bruijn variables; $\neg$, $\wedge$, $\vee$, $\leftrightarrow$ and $\exists$ are
defined. The proof system is a Hilbert system: the propositional axioms
$K$ and $S$ and double negation elimination, modus ponens,
generalization, instantiation of $\forall$ by a variable,
$\forall(\varphi{\uparrow}\to\psi)\to\varphi\to\forall\psi$ where $\varphi{\uparrow}$ is $\varphi$ with its
variables shifted, reflexivity of equality, and substitutivity of
equality in atomic formulas. Nonlogical axioms are open formulas, their
free variables being parameters. $\Con(T)$ says that $T$ does not prove
$\bot$. Satisfaction $M\models\varphi[e]$ is defined for a class $M:\V\to\Prop$
and an environment $e:\N\to\V$, by recursion on $\varphi$ into $\Prop$, with
$\in$ and $\approx$ for the atoms and quantifiers ranging over $M$. Soundness
(\lean{soundness}): if every axiom of $T$ holds in $M$ under every
environment in $M$, so does every theorem; so a theory with a nonempty
class model is consistent. Double negation elimination is where
soundness uses excluded middle.

\paragraph{The axioms (\lean{ZF}).}
Extensionality; Foundation, as the existence of an $\in$-minimal element
of a nonempty set; Pairing, Union and Power set, in the form that a
set including the required one exists; Infinity, as a set containing an
empty set and closed under $y\mapsto y\cup\{y\}$; Separation for every formula
$\psi$, with arbitrary parameters,
$\exists y\,\forall z\,(z\in y\leftrightarrow z\in x\wedge\psi)$; and Replacement for every
formula $\psi$, with arbitrary parameters,
\[\forall x{\in}a\ \forall y\,y'\,(\psi(x,y)\to\psi(x,y')\to y=y')\ \to\ \exists b\ \forall y\,
  (\exists x{\in}a\ \psi(x,y)\to y\in b).\]
For each axiom a lemma states what its satisfaction in a class model
means in ordinary notation (\lean{sat\_ext}, \dots, \lean{sat\_repl}),
as a check that the de Bruijn indices say what is intended.

\paragraph{Models (\lean{ZFModel}).}
A class $M$ is a model of $\ZF$ if it is transitive, contains $\emptyset$ and
$\omega$, is closed under pairs, unions, power sets and subsets given by
arbitrary predicates, and for every formula $\psi$, parameters in $M$ and
$a$ in $M$ such that $\psi$ is functional on $a$ in the sense of $M$, some
$b$ in $M$ contains the values.

\begin{lemma}[\lean{V\_model}]
If every functional relation on a set has an image, the class of all
sets is a model of $\ZF$.
\end{lemma}

For the other alternative let $I(\eta,q,x,y)$ hold if $q\approx\langle\ulcorner\psi\urcorner,
\langle k,\langle e_0,\dots,e_{k-1}\rangle\rangle\rangle$ for a formula $\psi$ with free variables
below $k+2$ and a code $\ulcorner\psi\urcorner$ of it as a hereditarily finite set,
$x,y\in V_\eta$, and $V_\eta\models\psi[x,y,e_0,\dots,e_{k-1}]$ (\lean{ISat}). Codes and
tuples are injective up to $\approx$, so the code $q$ determines the
relation.

\begin{lemma}[\lean{Vl\_model}]\label{lem:vrho}
Assume excluded middle. If $\rho$ is an ordinal that is not $0$, not a
successor, not $\omega$, and not reachable for this $I$, then $V_\rho$ is a model
of $\ZF$.
\end{lemma}
\begin{proof}
By trichotomy $\omega\in\rho$ and $\rho$ is closed under successor. So finitely
many elements of $V_\rho$ have ranks below a common element of $\rho$, and
$V_\rho$ is closed under the operations, which raise ranks by a finite
amount. For Replacement let $\psi$ have its free variables below $k+2$, let
$e_0,\dots,e_{k-1}$ and $a$ be in $V_\rho$, with $\psi$ functional on $a$ in
$V_\rho$, and let $q$ be the code. The ranks of $q$ and $a$ are below some
$\nu\in\rho$: the code of $\psi$ has finite rank, and a tuple has rank a finite
amount above its components. The pair $(q,a)$ is functional on $a$, with
values in $V_\rho$. As $\rho$ is not reached from $\nu$, the pair is not
unbounded, so by excluded middle there is $\zeta\in\rho$ such that no value has
rank at least $\zeta$, and $b:=V_\zeta$ is as required.
\end{proof}

\begin{proof}[Proof of Theorem~\ref{thm:main}]
By Theorem~\ref{thm:dich} for this $I$, the two lemmas, and soundness.
\end{proof}

\paragraph{One model (\lean{HG}, \lean{hg\_model}).}
The two alternatives can be merged into a single model whose definition
involves neither a case distinction nor the ordinal $\rho$. Call an
ordinal \emph{good} if it is $0$, a successor, $\omega$, or reachable, and
\emph{hereditarily good} if every ordinal up to it is good; this is the
class $C$ of Lemma~\ref{lem:worldlyrule}. Let $H$ be the class of the
sets whose rank is hereditarily good. If every ordinal is good then $H$
is the class of all sets, and otherwise it is $V_\rho$ for the least
ordinal $\rho$ that is not good, but one does not need to know which.

\begin{theorem}[\lean{hg\_model}]
Assume excluded middle. Then $H$ is a model of $\ZF$.
\end{theorem}
\begin{proof}
$H$ is transitive, and it is closed under the set operations because
they raise ranks by at most one and the successor of a hereditarily good
ordinal is hereditarily good. For Replacement let $\psi$, with parameters
in $H$, be functional in $H$ on $a$ in $H$. The ranks of the values are
hereditarily good, so by Lemma~\ref{lem:worldlyrule} and
Theorem~\ref{thm:repl} they form a set, and their strict supremum $\theta$
exists; every ordinal below $\theta$ is hereditarily good. If $\theta$ is good,
then $V_\theta$ is in $H$ and contains the values. If $\theta$ were not good,
then, the hereditarily good ordinals being closed downwards, they would
be exactly the elements of $\theta$; so $H$ would be $V_\theta$, satisfaction in
$H$ would be satisfaction in $V_\theta$, and Lemma~\ref{lem:vrho} would
bound the values inside $V_\theta$. (In fact this case is contradictory: the
values would have ranks cofinal in $\theta$ and also bounded below it.)
\end{proof}
\section{The formalization}\label{sec:lean}

The development is a Lean~4 package with no dependencies. We did not
use Mathlib \cite{mathlib}, which has the sets-as-trees construction,
because its proofs use the axiom of choice freely and Lean's report of
the axioms used by a theorem would then say nothing. The files, in order
of dependency:
\begin{center}
\begin{tabular}{lrl}
\lean{PSet} & 146 & $\V$, $\approx$, $\in$, $\in$-induction, the basic operations, guards\\
\lean{Mat} & 149 & labels, paths, the recursion $F$, coherence, Lemma~\ref{lem:mat}\\
\lean{Repl} & 113 & the glued assignment, Theorem~\ref{thm:repl}\\
\lean{Rule} & 100 & Lemma~\ref{lem:rule}\\
\lean{Pair}, \lean{Ord}, \lean{VLevel} & 316 & pairs, ordinals, rank, power set, $V_x$, $\omega$, $D$\\
\lean{Worldly} & 291 & the rule, Lemma~\ref{lem:worldlyrule}, Theorem~\ref{thm:dich}\\
\lean{Fml}, \lean{Proof} & 283 & formulas, satisfaction, codes, the proof system, soundness\\
\lean{ZF}, \lean{ZFRead} & 275 & the axioms, models, readback of the axioms\\
\lean{FirstOrder} & 263 & $I$, Lemma~\ref{lem:vrho}, Theorem~\ref{thm:main}\\
\lean{Uniform} & 131 & the single model $H$\\
\lean{Stable}, \lean{Negative} & 123 & $\neg\neg\EM$ suffices, Section~\ref{sec:em}
\end{tabular}
\end{center}
Excluded middle is never an axiom. It is a hypothesis
\lean{(em :\ $\forall$ p :\ Prop, p $\vee$ $\neg$p)} of the theorems that use it, which are:
trichotomy of ordinals and what depends on it (the characterization of
$V_x$ by rank, the least $\nu$), the case distinctions in
Lemma~\ref{lem:worldlyrule} and Theorem~\ref{thm:dich}, ``not unbounded
implies bounded'' in Lemma~\ref{lem:vrho}, and the classical connectives
in the semantics of formulas. Lemma~\ref{lem:mat},
Theorem~\ref{thm:repl} and Lemma~\ref{lem:rule} do not use it.

Lean reports that \lean{con\_ZF} depends on no axioms, and the source
asserts this with \lean{\#guard\_msgs}. Propositional extensionality is easy
to use by accident in Lean, and we record the three ways it entered: rewriting
with an equivalence (\lean{rw} with an \lean{Iff}); case analysis on an
equation between two distinct numerals other than $0$; and the library
lemmas about \lean{max} on natural numbers. We avoided them by
composing equivalences explicitly, by \lean{decide}, and by a
three-line definition of the maximum.

The universe level of $\V$ is a parameter throughout, and the main
theorem is proved at level $0$, where $\V:\Type_1$. The recursion $F$ is
an application of \lean{Acc.rec} with carrier the type of paths and
motive the constant $\V$, both in $\Type_1$.

\section{How much excluded middle is needed}\label{sec:em}

Excluded middle is the only hypothesis of Theorem~\ref{thm:main}. We
report what we know about removing it. The results in this section that
are marked as formalized are proved in Lean with no axioms; some of them
are on the branch \lean{no-em} of the repository, which carries out the
whole construction over the stable relations described below, and one
is due to Cameron Zwarich \cite{zwarich26}.

\paragraph{Double negation suffices.}
$\Con(\ZF)$ is a negation, so it follows from $\neg\neg\EM$, the double
negation of excluded middle. This is equivalent to the double negation
shift $(\forall x.\,\neg\neg A\,x)\to\neg\neg\forall x.\,A\,x$ for predicates on
arbitrary types (\lean{con\_ZF\_of\_not\_not\_em}, \lean{con\_ZF\_of\_dns}).
The question is therefore whether the type theory proves $\Con(\ZF)$
with no hypothesis at all.

\paragraph{The stable reading of sets.}
Read bisimulation and membership with $\neg\neg\exists$ in place of $\exists$ at
every level of the trees:
\[x\approx y:\iff\forall i\,\neg\neg\exists j.\ x_i\approx y_j\ \wedge\ \forall j\,\neg\neg\exists i.\ x_i\approx y_j,
\qquad x\in y:\iff\neg\neg\exists j.\ x\approx y_j.\]
This is the device of the negative interpretations of $\ZF$ in
intuitionistic set theory \cite{friedman73,powell75}, applied to the
trees.
Every statement built from these by $\forall$, $\to$, $\wedge$ and $\neg$ is then
\emph{stable}, that is, it follows from its double negation, and a stable
goal may be proved classically: a case distinction on any proposition,
and the use of a doubly negated hypothesis as if it held, are available
without excluded middle. Extensionality holds, and $\in$-induction holds
for stable predicates. On the branch \lean{no-em} the whole construction
is redone in this reading, and every use of excluded middle except one
disappears.
\begin{itemize}
\item The theory of ordinals goes through, with trichotomy in the form
  $\neg\neg(a\in b\vee a\approx b\vee b\in a)$, and so do rank, the levels $V_x$
  and their characterization by rank.
\item Satisfaction of formulas built from $\bot$, $\to$ and $\forall$ is stable,
  so double negation elimination is sound and the derived connectives
  have their classical reading.
\item The definability rule needs no case distinction between zero,
  successor and limit ordinals. An ordinal is \emph{good} if it is
  reachable as soon as it has an element. The label set is
  $D(G)=V_{\max(\rk G,\omega)+\omega}$, which contains $G$ and $\omega$ and is closed
  under pairs, and ``reached from $\nu$'' asks that the code and the domain
  of the function be elements of $D(\nu)$, so that no rank arithmetic is
  needed. That successors and $\omega$ are reachable is then a fact about
  first-order definability (a constant function, the identity on $\omega$),
  not a case of the rule.
\item There is a single model, the class $H$ of the sets of hereditarily
  good rank, as in Section~\ref{sec:zf}. In the Replacement step, that the
  supremum $\theta$ of the ranks of the values is good is still proved by
  contradiction: otherwise $H$ would be $V_\theta$, and the function itself
  would reach $\theta$. Since being good is stable, this costs nothing here.
\end{itemize}
The materialization lemma and Theorem~\ref{thm:repl} were constructive to
begin with.

\paragraph{What is left: accessibility of membership.}
Accessibility is not a stable proposition, and it is needed as such,
because the bounding set is the value of the recursion at the root and
the recursion consumes an accessibility proof. In the stable reading
the proof of Lemma~\ref{lem:mat} gives the value of the recursion for
any accessibility proof, but no longer the accessibility itself. What
remains is isolated as follows. Say that a target assignment has the
\emph{descent condition} if the target of a child is an element of the
target of its parent.
\begin{enumerate}
\item The type of paths is needed as the carrier of the recursion,
  because a step must produce its children as terms and the targets are
  not terms; but it is not needed for the well-foundedness. If $\tau$ has
  the descent condition and $p$ has the target $t$, then $p$ is accessible
  as soon as $t$ is accessible for $\in$ (\lean{acc\_of\_desc}).
\item Let \lean{AccHyp} say that the root of every target assignment with
  the descent condition is not not accessible. Then
  \lean{con\_ZF (hacc :\ AccHyp) :\ Con ZF}, with no axioms.
\item \lean{AccHyp} is equivalent to the statement that membership is not
  not well-founded, $\neg\neg\forall x{:}\V.\ \Acc_\in x$ (\lean{accHyp\_iff\_mem\_wf}):
  for the converse, apply it to the assignment whose paths are the
  descending $\in$-chains of sets.
\end{enumerate}
So, over the type theory with no axioms,
\[\neg\neg\EM\ \Longrightarrow\ \neg\neg\forall x{:}\V.\ \Acc_\in x\ \Longrightarrow\ \Con(\ZF),\]
where $\in$ is the stable membership, and the middle statement is
$\in$-induction for predicates that need not be stable. Structural
induction on trees gives accessibility of every $x_j$, and $\Acc$ respects
$\approx$; but $y\in x$ only says that it is not the case that no $j$ has
$y\approx x_j$, so the step gives $\neg\neg\Acc_\in y$ and not $\Acc_\in y$.

\paragraph{Markov's principle.}
Cameron Zwarich observed \cite{zwarich26} that accessibility of
membership has logical strength of its own: if $\omega$ is accessible for
the stable membership, Markov's principle holds,
\[\bigl(\forall n.\,P\,n\vee\neg P\,n\bigr)\to\neg\neg\exists n.\,P\,n\to\exists n.\,P\,n,\]
with the witness even in $\Type$ when $P$ has a decision procedure. The
argument is the mechanism of this paper in miniature. For $k\in\N$ let
\[s_k:=\textstyle\bigcup_{d\in\N}\guard{d}{\text{$d$ is the least $e$ with $P\,(k+e)$}},\]
a term, which is bisimilar to the remaining length of the search when it
exists. Given $\neg\neg\exists n.\,P\,n$, the set $s_0$ is an element of $\omega$, and if
$P\,k$ fails then $s_{k+1}\in s_k$. A recursion on the accessibility of
$s_0$ that decides $P\,k$ at each step then finds the witness. The guard
lets a term denote a number that is only classically determined, and
honest accessibility runs an unbounded search along it. Consequently
\lean{AccHyp} implies the double negation of Markov's principle for all
decidable predicates on $\N$ at once, and the double negation shift for
families of decidable searches indexed by any type.

This locates the hypothesis between two known principles:
\[\neg\neg\EM\ \Longrightarrow\ \lean{AccHyp}\ \Longrightarrow\ \neg\neg\,\mathrm{MP},\]
where $\mathrm{MP}$ is Markov's principle for all decidable predicates. Markov's
principle is not provable in Martin-L\"of type theory
\cite{coquandmannaa16}, and we expect, without having checked it, that
the same holds for Lean's core; in that case the honest statement
$\forall x.\,\Acc_\in x$ is not provable. Whether the doubly negated form, or
$\Con(\ZF)$ itself, is provable with no hypothesis is open. A proof of the
latter would have to avoid accessibility of membership, or use a
different source of large elimination.

\paragraph{Positive membership.}
With the membership of Section~\ref{sec:sets}, with a positive existential, $\in$ is well-founded outright by structural induction, and the
descent condition gives accessibility of the root directly. The uses of
excluded middle are then elsewhere, and two of them are essential to the
construction as it stands.
\begin{itemize}
\item The target of the child $\la$ is the \emph{least} ordinal from
  which the node's target is reached, a minimization over an undecidable
  property. Its role is to provide data from which the labels of the
  other children are enumerated, and its target must be an element of
  the node's target, determined by the node's target alone. The
  supremum of all such ordinals is the node's target itself, and a
  relation with several targets per path breaks the accessibility
  argument, since the target of a child descends from a different
  target of the parent than the one the induction is about. Taking an
  honest subset instead of a member leads to chains whose accessibility
  is again an unbounded search. We expect that the minimization can be
  avoided by growing the pool of labels in $\omega$ rounds, closing under
  the ordinals definable from the current pool, but we have not checked
  this.
\item That the supremum $\theta$ of the ranks of the values is good is
  proved by contradiction. The function at hand is definable over the
  model, while reachability asks for one definable over $V_\theta$; they
  agree when the model is $V_\theta$, which is the case distinction. Making
  reachability refer to the model itself is circular, and satisfaction
  in a class is not monotone in the class, so there is no fixed point to
  take. A positive inductive definition of the model gives several
  targets per path, and then the tree is not even classically
  well-founded: the ordinal $m$ is the supremum of the constant function
  with value $m-1$ on a one-element domain, so a single label admits every
  finite ordinal as a target. The classical proof avoids a reflection
  argument here, and reflection is what Replacement provides. For formulas
  whose quantifiers are bounded (with the power set as a primitive),
  whose meaning does not depend on the model, we expect the
  corresponding instances of Replacement to have a positive proof.
\end{itemize}
One could hope to build, with positive membership, a model of an
intuitionistic set theory, and to obtain $\Con(\ZF)$ from the known
reduction of $\ZF$ to it, just as we defer the construction of $L$ for
$\ZFC$ to the literature. Intuitionistic set theory comes in two forms,
which differ in how Replacement is stated. $\mathrm{IZF}$ has the schema of
Collection: if $\forall x{\in}a\,\exists y\,\varphi(x,y)$, some set contains a witness for
each $x\in a$. $\mathrm{IZF}_R$ has the schema of Replacement, for functional
$\varphi$ only. Classically the two are equivalent; intuitionistically
$\mathrm{IZF}_R$ does not prove Collection \cite{friedmanscedrov85}. Friedman's translation
\cite{friedman73} gives $\Con(\mathrm{IZF})\to\Con(\ZF)$ for the form with
Collection. This does not help with the two obstacles above: both already
arise in proving Replacement for a functional relation, so we do not know
how to build a model even of $\mathrm{IZF}_R$ with positive membership. And
even if they were overcome, the materializing recursion could only
produce $\mathrm{IZF}_R$, not $\mathrm{IZF}$. It needs the target of each child to be
determined by the node, and for a functional relation the target is the
rank of the value. For a relation that is not functional the only
canonical choice would be the least rank of a witness, which is again a
minimization. $\Con(\ZF)$ would then need a reduction of $\ZF$ to
$\mathrm{IZF}_R$. Whether $\mathrm{IZF}_R$ and $\mathrm{IZF}$ have the same strength is a
longstanding open problem, and it has been conjectured that $\mathrm{IZF}$
proves the consistency of $\mathrm{IZF}_R$ \cite{rathjen05,moczydlowski08}. If
so, then by the second incompleteness theorem $\mathrm{IZF}$ does not prove
$\Con(\mathrm{IZF}_R)\to\Con(\ZF)$, and the reduction would have to use more
than $\mathrm{IZF}$.

\begin{question}
Does the type theory with no axioms prove $\neg\neg\forall x{:}\V.\ \Acc_\in x$ for the
stable membership? Does it prove $\Con(\ZF)$?
\end{question}

\section{Discussion}\label{sec:discussion}

\paragraph{What the first alternative says.}
In the standard model of the type theory \cite{werner97,leewerner11}, where $\Type_0$ is
$V_\kappa$ for an inaccessible $\kappa$ and propositions are truth values,
$\V$ is $V_\kappa$ and the first alternative of
Theorem~\ref{thm:main} holds, as does the second. The first alternative
is Replacement for arbitrary propositional relations, the second-order
form, so where it holds the height of $\V$ is regular with respect to
everything the type theory can express. We do not know whether the type
theory proves it outright. The proof gives it from a weaker
hypothesis than the one we used: it suffices that for one relation $I$,
which may be any proposition and need not be first-order satisfaction,
every limit ordinal above $\omega$ is reachable.

\paragraph{More universes.}
With a third universe there is a type $\V_1$ of trees with index types
in $\Type_1$, which contains a set $\kappa_0$ of the ranks of the elements
of $\V$. We expect that the recursion of Section~\ref{sec:mat}, run on
paths of $\V$ with a relation that mentions a parameter from $\V_1$,
shows that $\kappa_0$ is a regular cardinal of $\V_1$, hence inaccessible
there, and so on up the hierarchy, which would give $\ZF$ with $n$
inaccessible cardinals from $n+2$ universes, matching the strength known
with choice \cite{werner97,carneiro19,kirstsmolka19}. This is not checked.

\begin{question}
Is the type theory with $\omega$ universes and excluded middle, without
choice, equiconsistent with the same theory with choice, that is, with
$\ZFC$ with $n$ inaccessible cardinals for every $n$?
\end{question}

\paragraph{Small carriers.}
If $X$ is a type in $\Type_0$ and $R$ a relation on it, the accessible
part of $R$ is a well-founded relation on a small type. Its height $\alpha$
injects into the power set of $X$, by sending $\xi<\alpha$ to the set of
elements of height $\xi$, so it is below the Hartogs number of
$X\to\Prop$; and the Hartogs number of a small type $Y$ is an ordinal of
$\V$ given by a term (it is the set of the order types of the
well-orderings of subsets of $Y$, each of which is computed by recursion
on its own accessibility, on a small carrier). So a recursion
on $\Acc_R$ explores a tree that is bounded in advance, and the argument
of this paper has no analogue. The restriction is not artificial: the
well-founded recursion used in ordinary mathematics and programming is
mostly on small carriers.

\begin{question}\label{q:small}
Does $\ZF$ prove the consistency of the type theory with excluded
middle, without choice, in which $\Acc$ is restricted to carriers in
$\Type_0$?
\end{question}

\paragraph{Other systems.}
The argument uses nothing specific to Lean: an impredicative $\Prop$,
inductive types, and large elimination for accessibility are also
present in Rocq, where the sets-as-trees have been
studied extensively \cite{werner97,barras10}. We expect the development
to port directly, with excluded middle as the only assumption.

\paragraph{Acknowledgements.}
The author thanks Lo\"{\i}c Pujet and Cameron Zwarich for our discussions
on the consistency strength of $\CIC$ and its fragments, which prompted
this work.

\paragraph{AI disclosure.}
This work was carried out in collaboration with Anthropic's Claude
models (Claude Fable~5.1 and Claude Opus~5), used through Claude Code.
The research direction, the questions pursued, and the decisions about
what to prove and how to state it were the author's; several of the
constructions were proposed by the model in the course of the work,
among them the materializing recursion of Section~\ref{sec:mat} and the
definability rule of Section~\ref{sec:rule}, while others came from the
author, among them the single model of Section~\ref{sec:zf} and the
reduction of the remaining hypothesis to the accessibility of membership
in Section~\ref{sec:em}. The Lean development was written largely by
the model, with revisions by the author, and the text of this paper was
drafted by the model and revised by the author.

The main result does not rest on the model's reasoning: it is checked
by Lean, with the axioms used reported by the kernel. Several earlier
arguments produced during the work turned out to be wrong, and one gap
in a pen-and-paper proof was found only by formalizing it. The claims
of Sections~\ref{sec:em} and~\ref{sec:discussion} that are not
formalized, and the statements about the literature, have not been
checked in the same way; they are marked as expectations where we have
not verified them.

\bibliographystyle{plainurl}
\bibliography{references}

\begin{thebibliography}{10}

\bibitem{aczel78}
Peter Aczel.
\newblock The type theoretic interpretation of constructive set theory.
\newblock In Angus Macintyre, Leszek Pacholski, and Jeff Paris, editors, {\em Logic Colloquium '77}, volume~96 of {\em Studies in Logic and the Foundations of Mathematics}, pages 55--66. North-Holland, 1978.
\newblock \href {https://doi.org/10.1016/S0049-237X(08)71989-X} {\path{doi:10.1016/S0049-237X(08)71989-X}}.

\bibitem{barras10}
Bruno Barras.
\newblock Sets in {Coq}, {Coq} in sets.
\newblock {\em Journal of Formalized Reasoning}, 3(1):29--48, 2010.
\newblock \href {https://doi.org/10.6092/issn.1972-5787/1695} {\path{doi:10.6092/issn.1972-5787/1695}}.

\bibitem{carneiro19}
Mario Carneiro.
\newblock The type theory of {Lean}.
\newblock Master's thesis, Carnegie Mellon University, 2019.
\newblock URL: \url{https://github.com/digama0/lean-type-theory/releases/download/v1.0/main.pdf}.

\bibitem{coquandmannaa16}
Thierry Coquand and Bassel Mannaa.
\newblock The independence of {M}arkov's principle in type theory.
\newblock In {\em 1st International Conference on Formal Structures for Computation and Deduction (FSCD 2016)}, volume~52 of {\em LIPIcs}, pages 17:1--17:18. Schloss Dagstuhl--Leibniz-Zentrum f{\"u}r Informatik, 2016.
\newblock \href {https://doi.org/10.4230/LIPIcs.FSCD.2016.17} {\path{doi:10.4230/LIPIcs.FSCD.2016.17}}.

\bibitem{lean4}
Leonardo de~Moura and Sebastian Ullrich.
\newblock The {Lean} 4 theorem prover and programming language.
\newblock In {\em Automated Deduction (CADE 28)}, volume 12699 of {\em LNCS}, pages 625--635. Springer, 2021.
\newblock \href {https://doi.org/10.1007/978-3-030-79876-5_37} {\path{doi:10.1007/978-3-030-79876-5_37}}.

\bibitem{feferman65}
Solomon Feferman.
\newblock Some applications of the notions of forcing and generic sets.
\newblock {\em Fundamenta Mathematicae}, 56(3):325--345, 1965.
\newblock \href {https://doi.org/10.4064/fm-56-3-325-345} {\path{doi:10.4064/fm-56-3-325-345}}.

\bibitem{friedman73}
Harvey Friedman.
\newblock The consistency of classical set theory relative to a set theory with intuitionistic logic.
\newblock {\em Journal of Symbolic Logic}, 38(2):315--319, 1973.
\newblock \href {https://doi.org/10.2307/2272068} {\path{doi:10.2307/2272068}}.

\bibitem{friedmanscedrov85}
Harvey Friedman and Andre {\v{S}}{\v{c}}edrov.
\newblock The lack of definable witnesses and provably recursive functions in intuitionistic set theories.
\newblock {\em Advances in Mathematics}, 57(1):1--13, 1985.
\newblock \href {https://doi.org/10.1016/0001-8708(85)90103-3} {\path{doi:10.1016/0001-8708(85)90103-3}}.

\bibitem{griffor94}
Edward Griffor and Michael Rathjen.
\newblock The strength of some {M}artin-{L}{\"o}f type theories.
\newblock {\em Archive for Mathematical Logic}, 33(5):347--385, 1994.
\newblock \href {https://doi.org/10.1007/BF01278464} {\path{doi:10.1007/BF01278464}}.

\bibitem{gylterud18}
H{\aa}kon~Robbestad Gylterud.
\newblock From multisets to sets in homotopy type theory.
\newblock {\em Journal of Symbolic Logic}, 83(3):1132--1146, 2018.
\newblock \href {https://doi.org/10.1017/jsl.2017.84} {\path{doi:10.1017/jsl.2017.84}}.

\bibitem{kirstsmolka18}
Dominik Kirst and Gert Smolka.
\newblock Large model constructions for second-order {ZF} in dependent type theory.
\newblock In {\em Certified Programs and Proofs (CPP 2018)}, pages 228--239. ACM, 2018.
\newblock \href {https://doi.org/10.1145/3167095} {\path{doi:10.1145/3167095}}.

\bibitem{kirstsmolka19}
Dominik Kirst and Gert Smolka.
\newblock Categoricity results and large model constructions for second-order {ZF} in dependent type theory.
\newblock {\em Journal of Automated Reasoning}, 63(2):415--438, 2019.
\newblock \href {https://doi.org/10.1007/s10817-018-9480-6} {\path{doi:10.1007/s10817-018-9480-6}}.

\bibitem{leewerner11}
Gyesik Lee and Benjamin Werner.
\newblock Proof-irrelevant model of {CC} with predicative induction and judgmental equality.
\newblock {\em Logical Methods in Computer Science}, 7(4), 2011.
\newblock \href {https://doi.org/10.2168/LMCS-7(4:5)2011} {\path{doi:10.2168/LMCS-7(4:5)2011}}.

\bibitem{moczydlowski08}
Wojciech Moczyd{\l}owski.
\newblock Normalization of {IZF} with {R}eplacement.
\newblock {\em Logical Methods in Computer Science}, 4(2), 2008.
\newblock \href {https://doi.org/10.2168/LMCS-4(2:1)2008} {\path{doi:10.2168/LMCS-4(2:1)2008}}.

\bibitem{powell75}
William~C. Powell.
\newblock Extending {G}{\"o}del's negative interpretation to {ZF}.
\newblock {\em Journal of Symbolic Logic}, 40(2):221--229, 1975.
\newblock \href {https://doi.org/10.2307/2271902} {\path{doi:10.2307/2271902}}.

\bibitem{rathjen05}
Michael Rathjen.
\newblock Replacement versus collection and related topics in constructive {Z}ermelo--{F}raenkel set theory.
\newblock {\em Annals of Pure and Applied Logic}, 136(1--2):156--174, 2005.
\newblock \href {https://doi.org/10.1016/j.apal.2005.05.010} {\path{doi:10.1016/j.apal.2005.05.010}}.

\bibitem{rathjen12}
Michael Rathjen.
\newblock Constructive {Z}ermelo--{F}raenkel set theory, power set, and the calculus of constructions.
\newblock In Peter Dybjer, Sten Lindstr{\"o}m, Erik Palmgren, and G{\"o}ran Sundholm, editors, {\em Epistemology versus Ontology: Essays on the Philosophy and Foundations of Mathematics in Honour of Per {M}artin-{L}{\"o}f}, volume~27 of {\em Logic, Epistemology, and the Unity of Science}, pages 313--349. Springer, 2012.
\newblock \href {https://doi.org/10.1007/978-94-007-4435-6_15} {\path{doi:10.1007/978-94-007-4435-6_15}}.

\bibitem{mathlib}
{The mathlib Community}.
\newblock The {Lean} mathematical library.
\newblock In {\em Certified Programs and Proofs (CPP 2020)}, pages 367--381. ACM, 2020.
\newblock \href {https://doi.org/10.1145/3372885.3373824} {\path{doi:10.1145/3372885.3373824}}.

\bibitem{constructiveepsilon}
{The Rocq Development Team}.
\newblock {ConstructiveEpsilon}: constructive choice from an existence proof over the natural numbers.
\newblock Module \texttt{Stdlib.Logic.ConstructiveEpsilon} of the Rocq standard library, \url{https://github.com/rocq-prover/stdlib}.
\newblock Contributed by Yevgeniy Makarov and Jean-Fran\c{c}ois Monin.

\bibitem{hottbook}
{The Univalent Foundations Program}.
\newblock {\em Homotopy Type Theory: Univalent Foundations of Mathematics}.
\newblock Institute for Advanced Study, 2013.
\newblock URL: \url{https://homotopytypetheory.org/book}.

\bibitem{werner97}
Benjamin Werner.
\newblock Sets in types, types in sets.
\newblock In {\em Theoretical Aspects of Computer Software (TACS 1997)}, volume 1281 of {\em LNCS}, pages 530--546. Springer, 1997.
\newblock \href {https://doi.org/10.1007/BFb0014566} {\path{doi:10.1007/BFb0014566}}.

\bibitem{zwarich26}
Cameron Zwarich.
\newblock Accessibility of membership implies {M}arkov's principle.
\newblock Lean formalization, file \texttt{ConZF/Markov.lean} on the branch \texttt{no-em-markov} of \url{https://github.com/zwarich/ConZF}, 2026.

\end{thebibliography}

\end{document}